\documentclass[12pt]{article}
\input epsf.tex

\usepackage{amsmath}
\usepackage{amsthm}
\usepackage{amsfonts}
\usepackage{amssymb}
\usepackage{graphicx}
\usepackage{latexsym}
\usepackage{upquote}
\usepackage[colorlinks=true]{hyperref}

\usepackage{amsmath,amsthm,amsfonts,amssymb}

\usepackage{epsfig}

\usepackage{amssymb}
\usepackage[all]{xy}
\xyoption{poly}
\usepackage{fancyhdr}
\usepackage{wrapfig}

\theoremstyle{plain}
\newtheorem{theorem}{Theorem}[section]
\newtheorem{proposition}[theorem]{Proposition}
\newtheorem{lemma}[theorem]{Lemma}

\theoremstyle{definition}
\newtheorem{definition}{Definition}
\theoremstyle{remark}
\newtheorem{remark}{Remark}

\advance \topmargin by -\headheight
\advance \topmargin by -\headsep
\evensidemargin \oddsidemargin
\def\cA{{\cal A}}

\def\cC{{\mathcal C}}

\def\cc{{\curvearrowright}}

\def\F{{\mathbb F}}

\def\cX{{\cal X}}

\def\IRS{{\textrm{IRS}}}

\def\ker{\textrm{{Ker}}}

\def\cK{\mathcal K}

\def\cL{{\mathcal L}}

\def\L{{\mathcal L}}

\def\N{{\mathbb N}}

\def\P{{\cal P}}

\def\Sub{{\textrm{Sub}}}

\def\chix{{\raise.5ex\hbox{$\chi$}}}

\def\Z{{\mathbb Z}}

\newcommand{\bz}{\mathbb{Z}}

\newcommand{\og}{\omega}

\newcommand{\Lc}{\mathcal{L}}
\newcommand{\Ac}{\mathcal{A}}
\newcommand{\bk}{\mathbb{F}}
\newcommand{\ccZ}{\mathcal{Z}}

\def\sm{\setminus}
\def\ss{\subset}
\DeclareMathOperator{\rk}{rk}
\DeclareMathOperator{\rank}{rank}

\def\Q{\mathcal{Q}}

\begin{document}
\title{Invariant random subgroups of lamplighter groups }
\author{Lewis Bowen\footnote{email:lpbowen@math.utexas.edu, NSF grant DMS-0968762 and NSF CAREER Award DMS-0954606} \\ University of Texas \\ $\ $ \\ Rostislav Grigorchuk \footnote{email:grigorch@math.tamu.edu, NSF GRANT  DMS-1207699} \\ Texas A\&M University\\ $\ $ \\ Rostyslav Kravchenko \footnote{email:rkchenko@gmail.com} \\ University of Chicago}

\maketitle

\begin{abstract}


Let $G$ be one of the lamplighter groups $({\mathbb{Z}/p\bz})^n\wr\mathbb{Z}$ and $\Sub(G)$ the space of all subgroups of $G$. We determine the perfect kernel and Cantor-Bendixson rank of $\Sub(G)$. The space of all conjugation-invariant Borel probability measures on $\Sub(G)$ is a simplex. We show that this simplex has a canonical Poulsen subsimplex whose complement has only a countable number of extreme points. If $F$ is a finite group and $\Gamma$ an infinite group which does not have property $(T)$ then the conjugation-invariant probability measures on $\Sub(F\wr\Gamma)$ supported on $\oplus_\Gamma F$ also form a Poulsen simplex. 
\end{abstract}
\noindent

\noindent
\tableofcontents

\section{Introduction}



There has been a recent increase in studies of the action by conjugation of  a locally compact group $G$ on its space $\Sub(G)$ of closed subgroups with special attention paid to the invariant probability measures of this action. For example, see  \cite{AGV12, Vo12, ABBGNRS11, ABBGNRS12, Ve11, Sa11, Gr11,
Ve10, Bo10, BS06, DS02, GS99, SZ94}. A random closed subgroup whose law is invariant under conjugation is called an {\em invariant random subgroup} or IRS. One of the goals of this research is to classify the distributions of IRS's for interesting groups. For example Stuck-Zimmer \cite{SZ94} proved that every ergodic IRS of a higher rank simple Lie group is induced from a lattice subgroup. A complete classification  of IRS's of the infinite symmetric group has been obtained by A. Vershik \cite{Ve11}. Many interesting results and applications of IRS's to semisimple Lie groups are announced in \cite{ABBGNRS11} and obtained in \cite{ABBGNRS12}. An important subspace of the space of Schreier graphs of a group of intermediate growth is described by Y.Vorobets \cite{Vo12}.  The study of  IRS's  is  closely  related  with  the  study  of  central  characters on   groups.  Interesting  results  in this  direction  were  obtained  recently  by  Dudko  and  Medynets  \cite{DM11,DM12} who  showed  in  particular  that  groups  of  Thompson's-Higman  type  have  only  trivial 
central  characters.

Let us begin with some preliminary observations and  notations. Let  $G$ be a discrete countable group and
$\Sub(G)$ be the set of all subgroups of $G$ equipped with the topology induced by the Tychonoff topology of the
product space $\{0,1\}^G$ (a subgroup $H$ is identified with its characteristic function
$\xi_H\in\{0,1\}^G$). Explicitly, a sequence of subgroups $H_i$ converges to $H$ if the event $\{g\in H_i\}$ stabilizes to the event $\{g\in H\}$ for any $g\in G$. $\Sub(G)$ is a closed subset of a compact metrizable space, so is itself a compact metrizable space. The group $G$ acts on $\Sub(G)$
continuously by conjugation $g \cdot H := gHg^{-1}$ (\emph{adjoint action}). Let $\IRS(G)$ denote the space of all
conjugation-invariant Borel probability measures on $\Sub(G)$ with the
weak*-topology. More precisely, we view $\IRS(G)$ as a subspace of the Banach dual of the space of continuous functions on $\Sub(G)$. From this viewpoint we see that it is a {\em simplex}. More precisely, it is a convex closed metrizable subset $K$ of a locally convex linear
space and every point $\xi$ in $K$ is the barycenter (i.e. $\xi=\int x d\mu$ ) of a unique probability
measure $\mu$ supported on the subset  of extreme points. It is a common abuse of language, which we will adopt here, to refer to a measure $\mu \in \IRS(G)$ as an IRS.


 A problem of general interest is: ``For  `interesting' groups, describe  the smallest sub-simplex of $\IRS(G)$ containing the ergodic continuous invariant measures on $\Sub(G)$'' (where continuous means that there are no points of positive mass).  Observe that  the atomic part of  any invariant measure is concentrated on a countable union of subgroups with
finite conjugacy classes.  Such subgroups are normal or have normalizers of finite index in $G$ and
respectively play the role of fixed points or periodic points for the action of $G$ on $\Sub(G)$. So it is natural
to replace the system $(\Sub(G),G)$  by  $(\mathcal{K}(G),G)$,   where $\mathcal{K}(G) \subset \Sub(G)$  is the {\em perfect
kernel} of $\Sub(G)$ (i.e., it is the unique perfect subset whose complement is countable). Observe that every continuous invariant measure is supported on $\mathcal{K}(G)$. We denote
the simplex of invariant probability measures supported on $\mathcal{K}(G)$  by $\IRS_{\ast}(G)$.

Recall that if $X$ is a topological space then its {\em Cantor-Bendixson derivative}, denoted $X'$, is its set of limit points (the complement $X \setminus X'$ consists of all isolated points in $X$). For an ordinal $\alpha$, the iterated Cantor-Bendixson derivative $X^{\alpha}$ is defined by
$X^{\alpha+1}=(X^{\alpha})'$, and
\[X^{\lambda}=\bigcap_{\alpha<\lambda} X^{\alpha}\]
if $\lambda$ is a limit ordinal.  If  $X$  is a Polish space   then  for some  countable ordinal
$\alpha_{\star} , X^{\alpha}=X^{\alpha_{\star}}$ for all $\alpha \geq \alpha_{\star}$ and
$X^{\alpha_{\star}}$ is the perfect kernel of $X$. The least ordinal $\alpha_{\star}$ with this
property  is called the {\em Cantor-Bendixson  rank} of $X$ and is denoted by $r_{CB}(X)$. 

It follows that $\Sub(G)$   has two important characteristics: its  Cantor-Bendixson rank $r_{CB}(G)$
  and  its perfect kernel $\mathcal{K}(G)$ which can be empty or  homeomorphic to a Cantor set depending on the cardinality of $\Sub(G)$:
  $\mathcal{K}(G)=\emptyset$ if and only if $\Sub(G)$ is countable.  For
 instance, finitely generated nilpotent groups or more generally polycyclic groups as well as Baumslag-Solitar  metabelian
 groups $BS(1,n), n\in \mathbb{Z}$   have  only countably many subgroups.

By contrast, the lamplighter groups $\mathcal{L}_{n,p} = ({\mathbb{Z}/p\mathbb{Z}})^n\wr\mathbb{Z}=\oplus_\mathbb{Z}({\mathbb{Z}/p\bz})^n\rtimes\mathbb{Z}$ (were $\wr$ denotes the wreath product)
 have uncountably many subgroups.  Therefore by the result of P. Kropholler  \cite{Kro85}  any
 solvable group $G$ of infinite rank also
 has   a nonempty perfect kernel $\mathcal{K}(G)$, as it necessarily has a section isomorphic to one
 of the groups $\mathcal{L}_{1,p}$  ($p$ prime) and hence has uncountably many subgroups. Metabelian  groups of finite rank also  may have uncountably many subgroups. We will show:
 

   
 \begin{theorem}\label{thm:main1}
  The perfect kernel of $\Sub(\mathcal{L}_{n,p})$ is $\Sub(\cA_{n,p})$ where $\mathcal{A}_{n,p}$ denotes the subgroup $\oplus_\mathbb{Z}({\mathbb{Z}/p\bz})^n<\mathcal{L}_{n,p}$.  Moreover, the Cantor-Bendixson rank of $\Sub(\mathcal{L}_{n,p})$  is the first infinite ordinal.
 \end{theorem}

The proof relies on \cite{GK12} which gives a comprehensive study of the lattice of subgroups of $\mathcal L_{n,p}$. 

To explain the next result, recall there are two distinguished classes of simplices: a {\em Poulsen 
 simplex} is any simplex whose extreme points form a dense subset while a {\em Bauer simplex} is any simplex whose extreme points form a closed subset. By \cite{LOS78}, there is a unique Poulsen simplex up to affine isomorphism. Moreover, its set of extreme points is homeomorphic to the Hilbert space $\ell^2$. For example, the space of all shift-invariant probability measures on $\{0,1\}^\Z$ forms a Poulsen simplex. On the other hand, there are uncountably many non-isomorphic Bauer simplices. For example, let $X$ be any compact metrizable space. Then the space $P(X)$ of all Borel probability measures on $X$ is a Bauer simplex and the subspace of extreme points of $P(X)$ is homeomorphic to $X$.
 
We will show:

\begin{theorem}\label{thm:main2}
 $\IRS_\ast (({\mathbb{Z}/p\mathbb{Z}})^n\wr\mathbb{Z})$ is a Poulsen simplex. 
\end{theorem}
 
 By comparison,  Stuck-Zimmer \cite{SZ94} shows that if $G$ is a lattice in a higher rank simple Lie group then $\IRS_\ast(G)$ is empty. It is an open question to describe $IRS_\ast(\textrm{SO}(n,1))$. It follows from A. Vershik's work \cite{Ve11} that if $G$ is the infinite symmetric group then $\IRS_\ast(G)$ is a Bauer simplex (moreover, its set of extreme points is explicitly parametrized). Also in \cite{Bo12} it is shown that if $G$ is a nonabelian free group then $\IRS_\ast(G)$ is Poulsen (because in this case $\cK(G)$ is the subspace of all infinite-index subgroups).

The idea of  the proof  of Theorem \ref{thm:main2} is quite simple  and is based on an analogy with the proof that for a finite alphabet  $A$  and integer $d\geq 1$,  the simplex   of  invariant measures for the
action of $\mathbb{Z}^d$  on $A^{\mathbb{Z}^d}$   by shift transformations is a Poulsen simplex.  Observe that
the proof of this fact presented in  the paper  of G. H. Olsen in \cite{Ol79}  (which is a good source of information about
the Poulsen simplex) has two mistakes, as firstly ergodicity is confused with the mixing property (bottom of
page 45) and secondly the statement  about the coincidence of distributions  of measures $\tilde{\mu}$ and
$\nu$ on cylinder sets based on a block $\Lambda(a)$ is incorrect  (very  near the end of the proof,  page 46).
The coincidence in fact holds only on blocks of much smaller size. The end of our proof of  Theorem
\ref{thm:main2}  shows how  to correct this mistake.  

If $N\lhd G$ is normal and $H<G$ is a subgroup then $H$ acts by conjugation on $\Sub(N)$ and one may study the simplex of $H$-invariant probability measures supported on $\Sub(N)$, denoted here by $\IRS_H(N)$. We consider the case $G=F\wr\Gamma=(\oplus_\Gamma F)\rtimes \Gamma$, with $N=\oplus_\Gamma F$ where $F$ is a nontrivial finite group and $\Gamma$ is a countably infinite group.
\begin{theorem}\label{thm:main3}
If $\Gamma$ does not have property (T) then $\IRS_\Gamma(N)$ is a Poulsen simplex. If $\Gamma$ has property (T) then $\IRS_\Gamma(N)$ is a Bauer simplex.
\end{theorem}

The proof of  Theorem \ref{thm:main3} is based on the proof of B. Weiss and E. Glasner  \cite{GW97}  that  the simplex of invariant probability measures for the shift-action of $G$ on $\{0,1\}^G$ is  a Poulsen simplex if $G$ does not have  Kazhdan's property (T)   and is a Bauer  simplex if $G$ has property (T). 

Observe that the set $\Sub(N)$ may be strictly smaller then the perfect kernel $\mathcal{K}(G)$ if $G/N$ has uncountably many subgroups. So our second result still leaves open the problem of describing the structure of the simplex $\IRS_*(F\wr\Gamma)$ of invariant measures supported on the perfect kernel.

\subsection{Plan of the paper}

In \S \ref{sec:prelim} we introduce notation and identify the perfect kernel of $\Sub(\cL_{n,p})$. In \S \ref{sec:Poulsen} we prove Theorem \ref{thm:main2} that $\IRS_\ast(\cL_{n,p})$ is a Poulsen simplex. In \S \ref{sec:Poulsen2} we prove Theorem \ref{thm:main3}. \S \ref{sec:topological} contains some general facts about IRS's. In \S \ref{sec:CB} we finish Theorem \ref{thm:main1} by computing the Cantor-Bendixson rank of $\Sub(\cL_{n,p})$. 

\section{Preliminaries}\label{sec:prelim}

In this section we set notation and identify the perfect kernel of $\Sub(\cL_{n,p})$. So fix a prime number $p$ and an integer $n\ge 1$. Let $$\mathcal{L}_{n,p}:=({\mathbb{Z}/p\bz})^n\wr\mathbb{Z}=\oplus_\mathbb{Z}({\mathbb{Z}/p\bz})^n\rtimes\mathbb{Z}.$$ Let $\mathcal{A}_{n,p}$ denote the subgroup $\oplus_\mathbb{Z}({\mathbb{Z}/p\bz})^n$ of $\mathcal{L}_{n,p}$. Let $R=\bk_p[x,x^{-1}]$, and define the structure of an $R$-module on $\Ac_{n,p}$ by setting $x\og$, $\og\in\Ac_{n,p}$ equal to the shift of $\og$ by $+1$, $(x\og)_i=\og_{i+1}$ for all $i$.

Elements of $\mathcal{L}_{n,p}$ are written as pairs $(v,s)$ where $v\in\mathcal{A}_{n,p}$ and $s\in\mathbb{Z}$. If $g=(v,s)$ and $h=(w,t)$ are elements of $\mathcal{L}_{n,p}$ then direct computation gives that $gh=(v+x^sw,s+t)$ and $g^{-1}=(-x^{-s}v,-s)$.

Let us introduce the notation $\varphi_t(x)=1+x+\dots+x^{t-1}$. The next result is \cite[Lemma 3.1]{GK12}.

\begin{lemma}\label{3_1}
\label{group}
Let $V$ be a subgroup of $\mathcal{L}_{n,p}$. Then it defines a triple $(s,V_0,v)$, where $s\in\mathbb{N}$ is such that $s\mathbb{Z}$ is the image of projection of $V$ on $\mathbb{Z}$, $V_0=V\cap\mathcal{A}_{n,p}$, satisfies $x^sV_0=V_0$, and $v\in\mathcal{A}_{n,p}$ is such that $(v,s)\in{}V$. The element $v$ is uniquely determined up to addition of elements from $V_0$. For $s=0$ one can choose $v=0$.

Conversely any triple $(s,V_0,v)$ with such properties gives rise to a subgroup of $\mathcal{L}_n$. Two triples $(s,V_0,v)$ and $(s',V'_0,v')$ define the same subgroup if and only if $s=s'$, $V_0=V'_0$ and $v+V_0=v'+V_0$.

Moreover, $V\subset{}V'$ if and only if $s'|s$, $V_0\subset{}V_0'$ and $v=\varphi_{s/s'}(x^{s'})v'\mod{}V_0'$.
\end{lemma}

Define $\pi_1: \Sub(\cL_{n,p}) \to \N$ by $\pi_1(V) = s$ where  $s\mathbb{Z}$ is the image of projection of $V$ on $\mathbb{Z}$. Define $\pi_2: \Sub(\cL_{n,p}) \to \Sub(\cA_{n,p})$ by $\pi_2(V)=V\cap \cA_{n,p}$.

\begin{lemma}\label{lemma:s}
There are only a countable number of subgroups $V$ of $\Lc_{n,p}$ with $\pi_1(V)>0$.
\end{lemma}
\begin{proof}
Let $V$ be such a subgroup, $s=\pi_1(V)$, $V_0=V\cap\Ac_{n,p}$, and $v\in\Ac_{n,p}$ be such that $(v,s)\in V$. The triple $(s,V_0,v)$ completely characterizes the group. Define a new $R$-module structure $N$ on $\Ac_{n,p}$ by $x*\og=x^s\og$. Then $V_0$ is an $R$-submodule of $N$. Note that since $s>0$, $N$ is isomorphic to $R^{ns}$ as an $R$-module (if $e_i$, $i\in[1,n]$ is the basis of $\Ac_{n,p}$ then the basis of $N$ is  $x^j e_i$, $i\in[1,n]$ and $j\in[0,s-1]$).

Note that there is only a countable number of submodules of $R^k$. Indeed, since $R$ is a PID and $R^k$ is finitely generated, any submodule is also finitely generated. Since $R$ is countable, $R^k$ is also countable, hence there is only a countable number of finite subsets of $R^k$. Hence there is only a countable number of finitely generated submodules.

Since $\Ac_{n,p}$ is countable,  we have that for each $s>0$ there is only a countable number of possible triples $(s,V_0,v)$.
\end{proof}


\begin{theorem}\label{perfect}
The perfect kernel of $\Sub(\cL_{n,p})$ is $\Sub(\cA_{n,p})$.
\end{theorem}
\begin{proof}
First we show that $\Sub(\cA_{n,p})$ does not have isolated points. Let $H$ be a subgroup of $\cA_{n,p}$ and consider it as a vector subspace over $\mathbb{F}_p$. There is a countable basis $f_i$ of $\cA_{n,p}$, indexed by a set $I$, and a partition $I=I_1\coprod I_2$ such that $\{f_i:~i\in I_1\}$ is a basis of $H$. At least one of the sets $I_1$ or $I_2$ is infinite. If $I_q$ ($q=1$ or $2$) is infinite, then the collection of subgroups $H_j$, $j\in I_q$, with $H_j$ being the subgroup generated by $\{f_i| i\in I_1\Delta\{j\}\}$ (by $\Delta$ we mean symmetric difference of two sets), does not contain $H$ but has $H$ as an accumulation point.
Now we have that $\Sub(\cA_{n,p})$ is a perfect subset of $\Sub(\cL_{n,p})$, hence the perfect kernel of $\Sub(\cL_{n,p})$ contains $\Sub(\cA_{n,p})$. By the previous Lemma we have that $\Sub(\cL_{n,p})\setminus \Sub(\cA_{n,p})$ is countable. The next Lemma then shows that perfect kernel must be equal to $\Sub(\cA_{n,p})$.
\end{proof}
\begin{lemma}
Let $X$ be a Cantor set and $X_1\subset X$ be a closed subset without isolated points such that $X\setminus X_1$ is at most countable. Then in fact $X=X_1$.
\end{lemma}
\begin{proof}
Indeed the set $U=X\setminus X_1$ is open, and any nonempty open subset of the Cantor set contains a small copy of the Cantor set and thus is uncountable.
\end{proof}
We skip the proof of the next lemma which is straightforward.
\begin{lemma}\label{lemma:pi}
The map $\pi_2$ is continuous.  The map $\pi_1$ is Borel and conjugation-invariant in the sense that $\pi_1(gVg^{-1})=\pi_1(V)$ for any $g\in \cL_{n,p}, V \in \Sub(\cL_{n,p})$.
\end{lemma}

\section{Conjugation-invariant measures}\label{sec:Poulsen}

In this section, we show that $\IRS_\ast(\cL_{n,p})$ is a Poulsen simplex thereby establishing Theorem \ref{thm:main2}. 

We use the superscript $e$ on a space of measures to denote the subspace of ergodic measures. For example, $\IRS^e(G)$ denotes the subspace of ergodic measures of $\IRS(G)$. Let $\IRS_s(\cL_{n,p}) \subset \IRS(\cL_{n,p})$ be those measures $\mu$ with $\mu( \pi_1^{-1}(s)) = 1$. By Lemma \ref{lemma:pi}, $\IRS^e(\cL_{n,p})$ is the disjoint union of $\IRS_s^e(\cL_{n,p})$ over $s \in \{0,1,2,3,\ldots\}$. By Theorem \ref{perfect}, $\IRS_0(\cL_{n,p})=\IRS_*(\cL_{n,p})$. 

\begin{lemma}
If $s>0$ and $\mu \in \IRS_s^e(\cL_{n,p})$ then $\mu$ is supported on a single finite conjugacy class. Therefore $\cup_{s=1}^\infty \IRS_s^e(\cL_{n,p})$ is countable.
\end{lemma}
\begin{proof}
By Lemma \ref{lemma:s}, $\mu$ is supported on a countable set. Because $\mu$ is ergodic, this implies it is supported on single conjugacy class $\cC \subset \Sub(\cL_{n,p})$. Therefore it is the uniform measure on $\cC$ which implies that $\cC$ is finite.
\end{proof}

Let $\IRS_\tau(\cA_{n,p})$ be the space of shift-invariant Borel probability measures on $\Sub(\cA_{n,p})$. Here $\tau$ is the shift on $\cA_{n,p}$ which corresponds to multiplication by $x$ if $\cA_{n,p}$ is considered as $R$-module. 

\begin{lemma}
The map $\pi_2$ pushes forward to an affine map $(\pi_2)_*:\IRS(\cL_{n,p}) \to \IRS_\tau(\cA_{n,p})$. The restriction of $(\pi_2)_*$ to $\IRS_0(\cL_{n,p})$ is an isomorphism.
\end{lemma}

\begin{theorem}\label{proposition:Poulsen}
$\IRS_\tau(\cA_{n,p})$ is a Poulsen simplex. Therefore, $\IRS_0(\cL_{n,p})=\IRS_*(\cL_{n,p})$ is also a Poulsen simplex.
\end{theorem}

\begin{proof}
Given an element $w \in \cA_{n,p}$, we let $w=(\ldots, w_{-1},w_0,w_1,w_2,\ldots)$ with each $w_k \in (\Z/p\Z)^n$. For $i\le j$ let $\cX^{[i,j]}$ be the subgroup of $\cA_{n,p}$ which consists of all $w$ with $w_k=0$ if $k\not\in[i,j]$.  For $i\le j$, let $P^{i,j}:\Sub(\cA_{n,p}) \to \Sub(\cX^{[i,j]} )$ be the intersection map $P^{i,j}(H) = H\cap \cX^{[i,j]}$. It is easy to see that $P^{i,j}$ is continuous.  It induces a map $P^{i,j}_*:\IRS(\cA_{n,p}) \to \IRS(\cX^{[i,j]})$.

Let $\mu \in \IRS_\tau(\cA_{n,p})$. It suffices to show that $\mu$ is a limit point of ergodic measures in $\IRS_\tau(\cA_{n,p})$. Let $m>0$ be an integer. Then $P^{0,m-1}_*\mu$ is a measure on $\Sub(\cX^{[0,m-1]})$. So the infinite direct product $(P^{0,m-1}_*\mu)^\Z$ is a measure on $\Sub(\cX^{[0,m-1]})^\Z$. Note that this measure is ergodic under the shift because it is Bernoulli.


Let $\Phi: (\Sub(\cX^{[0,m-1]}))^\Z \to \Sub(\cA_{n,p})$ be the map
$$\Phi(\ldots, H_{-1}, H_0, H_1,\ldots) = \bigoplus_{k \in \Z} x^{km}H_k$$
where each $H_k \in \Sub(\cX^{[0,m-1]})$.

This map intertwines the shift on $(\Sub(\cX^{[0,m-1]}))^\Z$ with the $m$-th power of the shift on $\Sub(\cA_{n,p})$. So $\Phi_*((P^{0,m-1}_*\mu)^\Z)\in \IRS(\cA_{n,p})$ is invariant and ergodic under $\tau^m$. Finally, let
$$\mu_m:= \frac{1}{m} \sum_{k=0}^{m-1} \tau^k \Phi_*((P^{0,m-1}_*\mu)^\Z).$$
This is a shift-invariant ergodic measure in $\IRS_\tau(\cA_{n,p})$. It is ergodic because if $E \subset \Sub(\cA_{n,p})$ is any measurable shift-invariant set then for each $k$,
$$\tau^k \Phi_*((P^{0,m-1}_*\mu)^\Z)(E) \in \{0,1\}$$
by ergodicity of $\Phi_*((P^{0,m-1}_*\mu)^\Z)$. By shift invariance, $\tau^k \Phi_*((P^{0,m-1}_*\mu)^\Z)(E) = \Phi_*((P^{0,m-1}_*\mu)^\Z)(E) $ for every $k$. So $\mu_m(E) \in \{0,1\}$ as required.

It now suffices to show that $\lim_{m\to\infty} \mu_m = \mu$. Note that if $j+k<m$ then $P^{0,j}_*\mu = P^{0,j}_* (\tau^k \Phi_*((P^{0,m-1}_*\mu)^\Z))$. Indeed, let $T$ be any subgroup of $\cX^{[0,j]}$. On the one element set $\{T\}\in \Sub(\cX^{[0,j]})$ the left hand side is equal to $\mu(\{H<\cA_{n,p} |~H\cap \cX^{[0,j]}=T\})$, and the right hand side is equal to the measure $\Phi_*((P^{0,m-1}_*\mu)^\Z)$ of the set $\{H<\cA_{n,p} |H\cap \cX^{[k,j+k]}=T\}$. If $[k,j+k]\subset[0,m-1]$ then it is further equal to the measure $P^{0,m-1}_*\mu$ of the set $\{H<\cX^{[0,m-1]} |H\cap \cX^{[k,j+k]}=T\}$ and hence to the measure $\mu$ of the set
\[
\{H<\cA_{n,p} |(H\cap \cX^{[0,m-1]})\cap \cX^{[k,j+k]}=T\}=\{H<\cA_{n,p} |H\cap \cX^{[k,j+k]}=T\},
\]
and by shift invariance of $\mu$ we have the equality.
Therefore, $\|P^{0,j}_*\mu_m - P^{0,j}_*\mu\|_1 \le 2j/m$ which implies $\lim_{m\to\infty} \mu_m = \mu$.


\end{proof}

\begin{remark}
If instead of $\cA_{n,p}$ we consider the compact group $\hat{\cA}_{n,p}=\prod_{\Z}(\Z/p\Z)^n$, and the set of its closed subgroups $\Sub(\hat{\cA}_{n,p})$, then by replacing in the above proof  the map $\Phi$ by
\[
\hat{\Phi}(\ldots, H_{-1}, H_0, H_1,\ldots) = \dots\times H_{-1}\times H_0\times H_1\times\dots
\]
we have  proof that the simplex $\IRS_\tau(\hat{\cA}_{n,p})$ of shift invariant measures on $\Sub(\hat{\cA}_{n,p})$ is Poulsen.
\end{remark}

\section{IRS's of $F\wr\Gamma$}\label{sec:Poulsen2}

In this section we determine the type of simplex of IRS's supported on $\Sub(\oplus_\Gamma F)$ for an arbitrary wreath product $F\wr\Gamma$ of a finite group $F$ and countable group $\Gamma$. In particular, we obtain another proof of Theorem \ref{proposition:Poulsen}.

 We think of $N=\oplus_\Gamma F$ as the set of all functions $\phi:\Gamma \to F$ such that $\phi^{-1}(f)$ is finite for every $f \in F \setminus \{e\}$. The multiplication in $N$ is defined  coordinate-wise: $(\phi \psi)(f) = \phi(f)\psi(f)$. The group $\Gamma$ acts on $N$ by $(\gamma \phi)(f) = \phi(\gamma^{-1} f)$. This action induces a $\Gamma$-action on $\Sub(N)$, the space of subgroups of $N$. Let $\IRS_\Gamma(N)$ be the simplex of $\Gamma$-invariant Borel probability measures on $\Sub(N)$ with the weak* topology. 

\noindent {\bf Theorem \ref{thm:main3}}. {\em
If $\Gamma$ does not have property (T) then $\IRS_\Gamma(N)$ is a Poulsen simplex. If $\Gamma$ has property (T) then $\IRS_\Gamma(N)$ is a Bauer simplex.}

\begin{proof}
By \cite[Theorem 1]{GW97}, if $\Gamma$ has property (T) and $\Gamma \cc X$ is any continuous action on a compact metrizable space, then the space of shift-invariant Borel probability measures on $X$ is a Bauer simplex. So we may assume $\Gamma$ does not have property (T).


It is well-known that the ergodic measures in $\IRS_\Gamma(N)$ are the extreme points of $\IRS_\Gamma(N)$. By Krein-Milman theorem the only closed convex subset of $\IRS_\Gamma(N)$ which contains the ergodic measures is $\IRS_\Gamma(N)$ itself. Therefore it suffices to show that the closure of the subset of ergodic measures in $\IRS_\Gamma(N)$ is convex. So let $\mu_1,\mu_2 \in \IRS_\Gamma(N)$ be ergodic measures. Hence it is enough to show that $(1/2)(\mu_1+\mu_2)$ is a limit of ergodic measures. 

For the definition of weak mixing and strong mixing we recommend \cite{Sch84}. A result mentioned there in combination with a result from \cite{BV93} shows, as is indicated in \cite{GW97}, Theorem 2, that there exists a weakly mixing measure-preserving action $\Gamma \cc (X,\lambda)$ on standard probability space $(X,\lambda)$ which is ergodic but not strongly ergodic. This implies the existence of a sequence $\{A_i\}_{i=1}^\infty$ of Borel subsets of $X$ such that
\begin{itemize}
\item $\lim_{i\to\infty} \lambda(A_i \vartriangle gA_i) = 0$ for every $g\in \Gamma$;
\item $\lambda(A_i)=1/2$ for every $i$,
\end{itemize}
see \cite{JS87}, lemma 2.3.

Let $\Psi_n: \Sub(N) \times \Sub(N) \times X \to \Sub(N)$ be the map 
$$\Psi_n(H_1,H_2, x) = (H_1 \cap \oplus_{J_n(x)} F) \times (H_2 \cap \oplus_{K_n(x)} F)$$ where 
$$J_n(x) = \{g\in \Gamma:~ x \in gA_n\}, \quad K_n (x) = \{g\in \Gamma:~ x \in X \setminus gA_n\}.$$
 Note that $\Psi_n$ is $\Gamma$-equivariant with respect to the diagonal action of $\Gamma$. 
 
 Let $\nu$ be an ergodic joining of $\mu_1$ and $\mu_2$. More precisely, $\nu$ is a $\Gamma$-equivariant Borel probability measure on $\Sub(H)\times \Sub(H)$ whose projections are $\mu_1,\mu_2$ respectively. For example, we may choose $\nu$ to be an ergodic component of the product $\mu_1\times \mu_2$.  Let $\eta_n:=(\Psi_n)_*(\nu \times \lambda) \in \IRS_\Gamma(N)$. Because $\Gamma \cc (X,\lambda)$ is weakly mixing, it follows that $\Gamma \cc (\Sub(N) \times \Sub(N) \times X, \nu \times \lambda)$ is ergodic (\cite{Sch84}, Proposition 2.2). So $\eta_n$ is ergodic.
 
 We claim that $\eta_n$ converges to $(1/2)(\mu_1+\mu_2)$ as $n\to\infty$. It suffices to show that for any finite set $U \subset \Gamma$ and any subgroup $T < \oplus_U F $, we have
$$\lim_{n\to\infty} \eta_n( \{ S \in \Sub(N):~ S \cap \oplus_U F = T\}) = (1/2)(\mu_1+\mu_2)( \{ S \in \Sub(N):~ S \cap \oplus_U F = T\}).$$
If $x \in \bigcap_{u \in U} uA_n$ then $U \subset J_n(x)$ which implies
$$\Psi_n(H_1,H_2,x) \cap \oplus_U F = H_1 \cap \oplus_U F.$$
Similarly, if $x \in \bigcap_{u \in U} u(X\setminus A_n)$ then $U \subset K_n(x)$ which implies
$$\Psi_n(H_1,H_2,x) \cap \oplus_U F = H_2 \cap \oplus_U F.$$
Therefore 
\begin{eqnarray*}
&&\eta_n( \{ S \in \Sub(H):~ S \cap \oplus_U F = T\})\\
&=& \nu\times\lambda(\{ (H_1,H_2,x):~\Psi_n(H_1,H_2,x) \cap \otimes_U F = T\})\\
&=& \nu\times\lambda(\{ (H_1,H_2,x):~x \in\bigcap_{u \in U} uA_n,~H_1 \cap \oplus_U F = T\})\\
&&+\nu\times\lambda(\{ (H_1,H_2,x):~x \in\bigcap_{u \in U} u(X\setminus A_n),~H_2 \cap \oplus_U F = T\})\\
&&+\nu\times\lambda(\{ (H_1,H_2,x):~x \notin (\bigcap_{u \in U} uA_n) \cup (\bigcap_{u \in U} u(X\setminus A_n)),~\Psi_n(H_1,H_2,x) \cap \oplus_U F = T\})\\
&=& \mu_1(\{ H_1:~ H_1 \cap \oplus_U F = T\}) \lambda( \bigcap_{u \in U} uA_n) \\
&&+  \mu_2(\{ H_2:~ H_2 \cap \oplus_U F = T\}) \lambda( \bigcap_{u \in U} u(X\setminus A_n)) \\
&&+\nu\times\lambda(\{ (H_1,H_2,x):~x \notin (\bigcap_{u \in U} uA_n) \cup (\bigcap_{u \in U} u(X\setminus A_n)),~\Psi_n(H_1,H_2,x) \cap \oplus_U F = T\}).
\end{eqnarray*}
Because $\{A_n\}_{n=1}^\infty$ is asymptotically invariant, 
$$\lim_{n\to\infty} \lambda( \bigcap_{u \in U} uA_n) = \lim_{n\to\infty} \lambda( \bigcap_{u \in U} u(X\setminus A_n))=1/2.$$
Also note
\begin{eqnarray*}
&&\nu\times\lambda(\{ (H_1,H_2,x):~x \notin (\bigcap_{u \in U} uA_n) \cup (\bigcap_{u \in U} u(X\setminus A_n)),~\Psi_n(H_1,H_2,x) \cap \oplus_U F = T\})\\
& \le&  1-  \lambda( \bigcap_{u \in U} uA_n) - \lambda( \bigcap_{u \in U} u(X\setminus A_n))
\end{eqnarray*}
tends to zero as $n\to\infty$. So the equations above imply
$$\lim_{n\to\infty} \eta_n( \{ S \in \Sub(N):~ S \cap \oplus_U F = T\}) = (1/2)(\mu_1+\mu_2)( \{ S \in \Sub(N):~ S \cap \oplus_U F = T\})$$
as required.
\end{proof}

\section{Topological and Borel structures on $\Sub(G)$}\label{sec:topological}
Let us make a few remarks of a general character concerning IRS's that may be useful.
\begin{proposition}
 If $N<G$ is a subgroup then the set $\Sub(G,N)$ of subgroups of $G$ containing $N$ is closed. 
\end{proposition}
\begin{proposition}(\cite{Vo12})\label{ginn}
The clopen sets $C_{A,B}=\{H\in \Sub(G)| A\subset H, H\cap B=\emptyset\}$, where $|A|,|B|<\infty$, form a basis for the topology of $\Sub(G)$. The sets $Q_g=\{N |g\in N\}$ generate the Borel $\sigma$-algebra of $\Sub(G)$.
\end{proposition}

If $\phi:G\to H$ is a homomorphism, then $N\mapsto\phi(N)$ maps $\Sub(G)$ to $\Sub(H)$ and $N\mapsto\phi^{-1}(N)$ maps $\Sub(H)$ to $\Sub(G)$. The image of the first map is $\Sub(\phi(G))$ and of the second map $\Sub(G,\ker(\phi))$. A special case of the second map is the intersection with some fixed subgroup. Indeed, if $i:G\to H$ is an inclusion, then $i^{-1}(N)=N\cap G$.  
\begin{proposition}\label{prop5.3}
The map $N\mapsto\phi^{-1}(N)$ is continuous, and the map $N\mapsto\phi(N)$ is Borel.
\end{proposition}
\begin{proof}
For the first part, note that  $\phi^{-1}(N_m)\to \phi^{-1}(N)$ as $m\to\infty$ means that $\{g\in \phi^{-1}(N_m)\}$ stabilizes to $\{g\in\phi^{-1}(N)\}$ for any $g\in G$. Now, $g\in\phi^{-1}(N_m)$ if and only if $\phi(g)\in N_m$. That is, the events $\{g\in\phi^{-1}(N_m)\}$ and $\{\phi(g)\in N_m\}$ are equal. It follows that $\phi^{-1}(N_m)\to\phi^{-1}(N)$ if and only if the event $\{\phi(g)\in N_m\}$ stabilizes to $\{\phi(g)\in N\}$ for any $g\in G$. However, we know that $N_m\to N$ and so $\{v\in N_m\}$ stabilizes to $\{v\in N\}$ for any $v\in H$ at all, in particular for any $v$ of the form $\phi(g)$.

For the second part, we just need to compute $\{N| \phi(N)\in Q_h\}$, where $Q_h=\{N'| h\in N'\}$, since $Q_h,h\in H$ generate the Borel $\sigma-$algebra of $\Sub(H)$ by Proposition \ref{ginn}. Now notice that $\phi(N)\in Q_h$ means that $h\in \phi(N)$, which is equivalent to $\phi^{-1}(h)\cap N\neq\emptyset$. Thus $\{N| \phi(N)\in Q_h\}=\cup_{g\in\phi^{-1}(h)}Q_g$ which is Borel, since $\phi^{-1}(h)$ is countable (if $\ker\phi$ is finite, then it is actually clopen, and hence $N\mapsto\phi(N)$ is continuous). 
\end{proof}

Thus given an IRS $\mu$ on $G$ we can get an induced IRS $\phi_*\mu$ on $\phi(G)<H$, and given an IRS $\nu$ on $H$ we can get induced IRS $(\phi^{-1})_*\nu$ on $G$, supported on $\Sub(G,\ker(\phi))$.

\section{The Cantor-Bendixson rank of $\Sub(\cL_{n,p})$}\label{sec:CB}

Here we will show that the Cantor-Bendixson rank of $\Sub(\cL_{n,p})$, denoted $r_{CB}(\cL_{n,p})$, is the first infinite ordinal $\omega$. To do this, we will deal with the topological space $\Sub(\cL_{n,p}) \sm\Sub(\cA_{n,p})$, and it will be convenient to give it a special name.
\begin{definition}
Denote $\Sub(\cL_{n,p}) \sm\Sub(\cA_{n,p})$ by $\ccZ$.
\end{definition} 
By theorem \ref{perfect} we know that the kernel of $\Sub(\cL_{n,p})$  is exactly $\Sub(\cA_{n,p})$. Since it is a closed subset of $\Sub(\cL_{n,p})$, any sequence of subgroups with limit   not in $\Sub(\cA_{n,p})$ consists, but for a finite number of them, of elements which are also not in $\Sub(\cA_{n,p})$. Therefore the Cantor-Bendixson rank of $\Sub(\cL_{n,p})$ is equal to the Cantor-Bendixson rank of $\ccZ$. In order to compute $r_{CB}(\ccZ)$  we will define the Cantor-Bendixson rank of a poset, construct a map $\ccZ$ onto a particular poset denoted by $\Q$ which will encode the convergence properties of $\ccZ$ in terms of the order relation on $\Q$. We will then prove that $r_{CB}(\ccZ) = r_{CB}(\Q)$ and that $r_{CB}(\Q)=\omega$.

Let $\P$ be any set with transitive relation $<$. An element $p\in \P$ is {\em minimal} if there does not exist $q\in \P$ with $q<p$.
\begin{definition}
Let $\P_0$ be the set of all minimal elements of $\P$, and $\bar{\P}_0=\P\sm \P_0$. Inductively, for any ordinal $\alpha$ let $\P_{\alpha+1}$ be the union of $\P_\alpha$ and the set of minimal elements of $\bar{\P}_{\alpha}$, and $\bar{\P}_{\alpha+1}=\P\sm \P_{\alpha+1}$. For a limit ordinal $\beta$, let $\P_\beta=\cup_{\alpha<\beta}\P_\alpha$, and $\bar{\P}_\beta=\cap_{\alpha<\beta}\bar{\P}_{\alpha}$
The Cantor-Bendixson rank of $\P$, $r_{CB}(\P)$, is defined as the minimal ordinal $\alpha$ such that $\bar{\P}_{\alpha}=\bar{\P}_{\alpha+1}$.
\end{definition}
\begin{lemma}\label{saved}
Let $\P$ be such that $\bar{\P}_{r_{CB}(\P)}=\emptyset$. If $\P'\ss\P$ with the relation on $\P'$ induced from $\P$, then $r_{CB}(\P')\leq r_{CB}(\P)$.
\end{lemma}
\begin{proof}
Note that $\P_0\cap \P'\ss \P'_0$, hence $\bar{\P'}_0\ss\bar{\P}_0$. Then by (transfinite) induction we obtain that $\bar{\P'}_\alpha\ss\bar{\P}_\alpha$ for any ordinal $\alpha$. It follows that if $\bar{\P}_\alpha=\emptyset$ then $\bar{\P'}_\alpha=\emptyset$.
\end{proof}

\begin{definition}
Let $\Q$ be the partially ordered set, with the set of elements $\N\times (\N\cup\{0\})$, and $(t',r')<(t,r)$ if and only if $t'|t$ and $t'r'<tr$.
\end{definition}
\begin{theorem}\label{thm:Q}
The Cantor-Bendixson rank of $\Q$ is $\omega$, where $\omega$ is the first infinite ordinal.
\end{theorem}
\begin{proof}
Let $\Q'$ be the set $\N\times(\N\cup\{0\})$ together with relation $(a,b)<(a',b')$ if $a|a'$ and $b<b'$. Clearly, the Cantor-Bendixson rank of $\Q'$ is $\omega$. The Cantor-Bendixson rank of the natural numbers with the ``less than" relation, $(\N,<)$, is also $\omega$. 

Note that we have relation-preserving inclusions $\N\to\Q$, $n\mapsto(2^n,1)$ and $\Q\to\Q'$, $(t,r)\mapsto(t,tr)$. So the theorem follows from lemma \ref{saved}.
\end{proof}

The rest of this section is devoted to showing:
\begin{theorem}\label{thm:Q-hard}
The Cantor-Bendixson rank of $\Sub(\cL_{n,p})$ is equal to the Cantor-Bendixson rank of $\Q$.
\end{theorem}

Theorems \ref{perfect}, \ref{thm:Q-hard} and \ref{thm:Q} immediately imply Theorem \ref{thm:main1}.


\subsection{A map $\Phi$ from $\ccZ$ to $\Q$}

In this section, we define a map from $\ccZ$ into $\Q$. Recall that $R=\F_p[x,x^{-1}]$. We identify $R^n$ with $\left(\F_p[x,x^{-1}]\right)^n$ and with the subgroup $\cA_{n,p}<\cL_{n,p}$. 
\begin{definition}\label{expp}
Let $M$ be an $R$-module and $U\ss M$ an abelian subgroup.  Let $e(U)$ be the minimal positive $e$ such that $x^eU=U$, if such $e$ exists, and $+\infty$ otherwise. Note that if $x^sU=U$ for some $s>0$ then $e(U)$ divides $s$.
\end{definition}

\begin{definition}
If $M$ is a finitely generated $R$-module then $\rk(M)$, the rank of $M$ is the maximal number of elements which freely generate a free submodule of $M$. If no free submodule of $M$ exists then $\rk(M):=0$.
\end{definition}
Note that since $R$ is principal ideal domain, all freely generating sets will have the same cardinality. This also implies that rank is additive ($\rk(M/M')=\rk(M)-\rk(M')$) and that $\rk(M)=0$ implies that $M$ is a sum of nontrivial factors of $R$, hence is finite (any nontrivial factor of $R$ is finite).

\begin{definition}
Let $U$ be an abelian subgroup of an $R$-module $M$. Suppose $x^mU = U$ for some integer $m\ge 1$. Then we can consider $U$ as an $R$-module via the rule $x\ast u :=x^mu$. Let $\rk_m(U)$ be the rank of $U$ considered as an $R$-module with this rule. For example, $\rk_m(R^n)=nm$. Define $r_{M,U}=ne-\rk_e(U)$ where $e=e(U)$ and $n=\rk(M)$. 
\end{definition}
Note that if $\rk_1(U)$ is defined then $U$ is in fact an $R$-submodule, hence $\rk_1(U)=\rk(U)$.

\begin{definition}\label{tv}
For each subgroup $V\ss \cL_{n,p}$ which does not lie in $R^n=\cA_{n,p}$ let $(s,U,v)$ be its triple. Define $t_V=s/e_U$ and $r_V=r_{R^n,U}$. 
\end{definition}
Note the difference between the notations $r_V$ and $r_{M,U}$. The first one implies that $V\in\ccZ$, while the notation $r_{M,U}$ implies that $M$ is an $R$-module of finite rank, and $U\ss M$ is an abelian subgroup with finite $e(U)$.


\begin{proposition}\label{thm:Q-key}
Define $\Phi: \ccZ \to \Q$ by $\Phi(V) = (t_V,r_V)$. Note that $\Phi$ is order preserving. Also, for any $q<\Phi(V) \in \Q$ there exists a sequence $\{V_m\}_{m=1}^\infty \subset  \ccZ$ such that $\Phi(V_m) = q$ for all $m$ and $V_m\to V$ as $m\to\infty$. Conversely if $\{V_m\}_{i=1}^\infty \subset \ccZ$ limits on $V \in \ccZ$ and $V_m \ne V$ for every $m$ then there exists  a subsequence $\{V_{m_k}\}_{k=1}^\infty$ and $q<\Phi(V)$ such that each $V_{m_k}$ is mapped to $q$. 
\end{proposition}
Because of this Proposition, we say that $\Q$ is an {\em encoding poset} for the topological space $\ccZ$. Observe that Theorem \ref{thm:Q-hard} is immediately implied by Proposition \ref{thm:Q-key}. The rest of the section is devoted to proving Proposition \ref{thm:Q-key}.

\subsection{Convergence in $\ccZ$}
In this section we study properties of convergent sequences in $\ccZ$.

\begin{theorem}\label{3_3}
Suppose $V$ is a subgroup of $\mathcal{L}_{n,p}$. Then either $V$ is a subgroup of $\cA_{n,p}$ or $V$ is isomorphic to $\mathcal{L}_{k,p}$ for some $k\geq 1$. In particular if $s>0$ then $V$ is finitely generated, and if $V$ is of finite index in $\mathcal{L}_{n,p}$ then $V$ is isomorphic to $\mathcal{L}_{ns,p}$ where $s$ is the projection of $V$ onto $\mathbb{Z}$. If $V$ is a finitely generated subgroup then $V$ is either an elementary $p$-group of finite rank or isomorphic to $\mathcal{L}_{k,p}$ for some $k \geq 1$.
\end{theorem}
\begin{proof}
This is \cite[Theorem 3.5]{GK12}.
\end{proof}

\begin{lemma}\label{lll1}
Suppose $V_m\to V$ in $\ccZ$. Then there is $m_0$ such that $V\ss V_m$ for all $m\geq m_0$.
\end{lemma}
\begin{proof}
Indeed, such $V$ is isomorphic to some $\L_{k,p}$ by Theorem \ref{3_3} and thus is finitely generated. So for some finite set of generators of $V$ there is $m_0$ such that for all $m\geq m_0$, $V_m$ contains this set of generators, and therefore $V\ss V_m$.
\end{proof}
\begin{lemma}\label{lll2}
Suppose $V_m\to V$ in $\Sub(\cL_{n,p})$. Then $V_m\cap\cA_{n,p}\to V\cap\cA_{n,p}$.
\end{lemma}
\begin{proof}
Follows from Proposition \ref{prop5.3}.
\end{proof}
\begin{lemma}\label{lll0.1}
Let $U\ss U_m$ be submodules of $R^n$, $U_m\to U$, and $\rk(U)=\rk(U_m)$. Then $U_m$ stabilizes to $U$.
\end{lemma}
\begin{proof}
Let $n'=n-\rk(U)$. Since $U\ss U_m$ the event $v\in U_m$ is equal to the event $v+U\in U_m/U$. So if $U_m\to U$ it follows that $U_m/U\to 0$ in $R^n/U\simeq R^{n'}\oplus R_0$, where $R_0$ is a finite $R$ module. Now, $\rk(U_m)=\rk(U)$ implies that $U_m/U$ is finite, and so $U_m/U\subset R_0$. A convergent sequence of subsets of a finite set stabilizes.
\end{proof}
\begin{lemma}\label{lll2.5}
Let $U$ be an additive subgroup of $R^n$, and $\{f_m\}_{m\geq 1}$ be the list of irreducible elements of $\F_p[x]$. Then $U_m=f_m U\to \{0\}$.
\end{lemma}
\begin{proof}
Let $v\in R^n$. Then $v\in U_m=f_m U$ if and only if each coordinate of $v$ is divisible by $f_m$. If $v$ is in infinitely many $U_m$ it means that each coordinate of $v$ is divisible by infinitely many irreducible elements of $R$, which implies that $v=0$.
\end{proof}




\begin{lemma}\label{lll3}
Suppose that $U_m$, $U$ are additive subgroups of $R^n$, $U\ss U_m$ and $e(U_m),e(U)$ divide $s>0$ for all $m$. Suppose that $U_m\to U$ (in the space of subsets of $R^n$). Then for any $v\in R^n$ we have that the subgroups $V_m$ with triples $(s,U_m,v)$ converge to subgroup $V$ with triple $(s,U,v)$.    
\end{lemma}
\begin{proof}
Note that $(w,t)\in V_m$ if and only if $s|t$ and $(w,t)(v,s)^{-t/s}\in U_m$. This event converges to $(w,t)(v,s)^{-t/s}\in U$, which is true if and only if $(w,t)\in V$.
\end{proof}

\begin{lemma}\label{count}
The number of submodules $M$ of $R^k$ such that $\dim_{\mathbb{F}_p}R^k/M=a$ is equal to $p^{ak}-p^{(a-1)k}$ for $a>0$ and to $1$ for $a=0$.
\end{lemma}
\begin{proof}
This is \cite[Lemma 3.8]{GK12}
\end{proof}

\begin{lemma}\label{lll2.75}
For any $n>0$, $b>0$ and $0<r\leq nb$ there is an additive subgroup $U\ss R^n$ such that $e(U)=b$ and $\rk_b(U)=r$. 
\end{lemma}
\begin{proof}
Consider first the case $r=nb$. For $a\in\mathbb{N}$ let $P_a$ be the set of all subgroups $U\ss R^n$ such that $x^aU=U$ and $\dim_{\F_p} R^n/U=1$. Then $|P_a|=p^{na}-1$ by lemma \ref{count}. We are going to show that the number of elements in $\cup_{a|b, a<b} P_a$ is strictly smaller than the number of elements in $P_b$ (note that $a|b$ implies $P_a\ss P_b$). Then any subgroup in $P_b$ which is not in any $P_a$, $a<b$ will have $e(U)=b$, and since $\dim_{\F_p} R^n/U$ is finite, $\rk_{b}(U)=nb$.

Let $q_1,\dots,q_m$ be all primes that divide $b$, and $a_i=b/q_i$. Then if $a|b$ and $a<b$, there is some $i$ such that $a|a_i$. Hence $\cup_{a|b, a<b} P_a=\cup_i P_{a_i}$. So it suffices to show that $\sum_i p^{na_i}<p^{nb}$. We have that,
\[
\sum_i p^{na_i-nb}=\sum_i (p^{1-q_i})^{na_i}\leq \sum_i 2^{1-q_i} < \sum_{s=1}^{\infty} 2^{-s}=1.
\]
For $r<nb$ note that we can choose a desired subgroup in each summand $R$ of $R^n$, and then add them together to get a subgroup of $R^n$ of desired rank (since rank is additive). Thus it suffices to give a proof for $n=1$. Note also that any element of $R$ can be written as $\sum_{i=0}^{b-1}f_i(x^b)x^i$, where $f_i\in \F_p[y,y^{-1}]$. If $r<b$ take $U$ to be the set of all $\sum_{i=0}^{r-1}f_i(x^b)x^i$. Obviously, if $s<b$ then $x^sU\neq U$ (if $s\geq r$, then $1\in U$ but $x^s\not\in U$, and if $s<r$ then $x^{r-s}\in U$, but $x^r\not\in U$). Thus $e(U)=b$. 

\end{proof}

\begin{lemma}\label{lll0.2}
\[
\rk_{be}(U)=b\rk_{e}(U).
\] 
\end{lemma}
\begin{proof}
Note that if $\rk_{e}(U)=m$ and $f_1,\dots,f_m$ is the basis of $U$ as $R$ module, then $\{x^jf_i|0\leq j<b, 1\leq i\leq m\}$ is the basis of $U$ considered as $R$ module with the action $x*u=x^bu$. Thus the rank of $U$ considered as $R$ module with this $R$ action is $b\rk_{e}(U)$. 
\end{proof}

\begin{lemma}\label{lll4}
Let $M$ be a finitely generated $R$ module, $n=\rk(M)$. Let $U$ be an additive subgroup of $M$ with $e(U)=e$, $r_{M,U}=ne-\rk_{e}(U)>0$, let $b>0$ and $r'<r_{M,U}b$. Then there exist a sequence $\{U_m\}$ such that $U\ss U_m\ss M$,  $U_m\neq U$, $e(U_m)=eb$, $U_m\to U$ (in the space of subsets of $M$), and $r_{M,U_m}= r'$.
\end{lemma}
\begin{proof}
First we reduce to the case when $e=1$, that is when $U$ is a submodule. Indeed, we can consider new module $M'$, with the set of elements $M$ and new action of $R$ given by $x*m=x^em$. Note that $M'$ is still finitely generated and $\rank (M')=ne$.

Now we reduce the Lemma to the case $U=0$. Indeed since $U\ss U_m$ we can construct $U_m/U\subset M/U$ and then pullback to $U_m\subset M$. Note that $\rk(M/U)=n-\rk(U)=r_{M,U}$. Also, $e(U_m/U)=e(U_m)=b$, hence using Lemma \ref{lll0.2}, $\rk_b(U_m)=b\rk(U)+\rk_b(U_m/U)$, and thus $r_{M,U_m}=\rk(M)b-\rk_b(U_m)=\rk(M/U)b-\rk_b(U_m/U)=r_{M/U,U_m/U}$.

So now we need to prove that given $b>0$ and $r'<nb$ there exist $U_m\ss M$ such that $e(U_m)=b$, $U_m\to 0$, and $r_{M,U_m}=r'$, that is $\rk_{b}(U_m)=nb-r'$. Note that $M$ is isomorphic to $R^n$ plus some finite module, so we may as well construct such a sequence $U_m$ in $R^n$. Also, it suffices to construct just one additive subgroup $U'\ss R^n$ such that $e(U')=b$ and $\rk_{b}(U')=nb-r'$. Indeed then we may let $U_m=f_m U'$ with $f_m$ as in Lemma \ref{lll2.5}. We will have that $e(U_m)=e(U')$, $\rk_{b}(U_m)=\rk_{b}(U')$, and $U_m\to 0$ by Lemma \ref{lll2.5}. Thus the proof is finished by Lemma \ref{lll2.75}.
 \end{proof}
 
\subsection{Properties of the map $\Phi$}
We can now prove Proposition \ref{thm:Q-key} (recall that numbers $t_V$ and $r_V$ were defined in Definition \ref{tv}):

\begin{lemma}\label{a1}
Suppose $V_m\to V$ in $\ccZ$. Then by passing to a subsequence one may assume that $t_{V_m}$ and $r_{V_m}$ are constant. Moreover, then $t_{V_m}| t_V$ and $t_{V_m}r_{V_m}\leq t_V r_V$. If the sequence $V_m$ does not stabilize to $V$ then the last inequality is strict.
\end{lemma}
\begin{proof}
Let $(s_m, U_m,v_m)$ be the triples of $V_m$, and denote $e_m=e(U_m)$. By Lemma \ref{lll1}, passing to a subsequence we may suppose that $V\ss V_m$. Thus, by Lemma \ref{3_1} $s_m|s$, and hence, passing to a subsequence we may suppose that $s_m$ is constant, say $s_m=s'$. Also by Lemma \ref{3_1}, $x^{s'} U_m=U_m$, and therefore by the remark after Definition \ref{expp}, $e_m| s'$. By passing to a subsequence we may assume that $e_m$ is constant, say $e'$. Finally, $r_{V_m}=ne'-\rk_{e'}(U_m)$, so $0\leq r_{V_m}\leq ne'$, therefore, by passing to a subsequence we may assume that $r_{V_m}$ is constant, say $r'$. Also, $t_{V_m}=s'/e'$ is constant, denote it by $t'$.

Denote $t=t_V$ and $r=r_V$. It is left to prove that $t'|t$ and $t'r'\leq tr$. By Lemmas \ref{lll1} and \ref{lll2}, $U_m\to U$ and $U\ss U_m$, thus in particular $U=\cap_m U_m$. Hence $x^{e'}U=U$, and by the remark after Definition \ref{expp}, $e|e'$. Let $b=e'/e$. Since $s'|s$ and $s=te$, $s'=t'e'=t'be$, we have that $t'b|t$. Thus $t'|t$. 

Since $U\ss U_m$, we have that $b \rk_{e}(U)=\rk_{e'}(U)\leq \rk_{e'}(U_m)$. Thus $r'=ne'-\rk_{e'}(U_m)\leq ne'-\rk_{e'}(U)=neb-b\rk_{e}(U)=br$. Also, $t'b|t$, thus $b\leq t/t'$. It follows that $r'\leq tr/t'$ and thus $t'r'\leq tr$. Notice that the equality holds if and only if $b=t/t'$ and $\rk_{e'}(U_m)=\rk_{e'}(U)$. Thus if the equality holds, $t'e'=t'eb=te$, and by Lemma \ref{lll0.1} $U_m$ stabilizes to $U$. Hence $V_m$ has triples $(te,U,v_m)$. Since $V\subset V_m$, we have by Lemma \ref{3_1} that $v_m=v\mod U$ and so, again by Lemma \ref{3_1}, $V_m=V$.
\end{proof}


The following Lemma gives a converse statement.
\begin{lemma}\label{a2}
Suppose $(t',r')<(t,r)$ in $\Q$. Let $V$ be a subgroup of $\cL_{n,p}$, not contained in $ R^n$, such that $t_V=t$ and $r_V=r$. Then there exists a sequence of subgroups $V_m$, not contained in $ R^n$, such that $t_{V_m}=t'$, $r_{V_m}=r'$, and $V_m\to V$, nonstabilizing. 
\end{lemma}
\begin{proof}
Let $(te,U,v)$ be a triple of $V$. Since $tr>t'r'$, we have that $r>0$. Let $b=t/t'$, so that $r'<br$. By Lemma \ref{lll4}, there are $U\ss U_m\ss R^n$ such that $e(U_m)=eb$, $r_{R^n,U_m}=r'$ and $U_m\to U$. Let $V_m$ be the subgroups defined by triples $(te, U_m,v)$. Then $V_m\to V$ by Lemma \ref{lll3}.
\end{proof}

Theorem \ref{thm:Q-key} follows immediately from Lemmas \ref{a1}, \ref{a2}.


{\small

\begin{thebibliography}{1000000}



\bibitem[AB+11]{ABBGNRS11} M. Abert, N. Bergeron, I. Biringer, T. Gelander, N. Nikolov, J. Raimbault and I. Samet, \textit{On the growth of Betti numbers of locally symmetric spaces}. Comptes Rendus Mathematique, Volume 349, Issues 15--16, August 2011,  831--835.

\bibitem[AB+12]{ABBGNRS12} M. Abert, N. Bergeron, I. Biringer, T. Gelander, N. Nikolov, J. Raimbault and I. Samet, \textit{On the growth of $L^2$-invariants for sequences of lattices in Lie groups}. arXiv:1210.2961


\bibitem[AGV12]{AGV12} M. Abert, Y. Glasner and B. Virag,  \textit{Kesten's theorem for Invariant Random Subgroups}. arXiv:1201.3399






 \bibitem[Bo10]{Bo10} L. Bowen, \textit{Random walks on coset spaces with applications to Furstenberg entropy}. To appear in Invent. Math.

\bibitem[Bo12]{Bo12} L. Bowen \textit{Invariant random subgroups of the free group}. arXiv:1204.5939.

\bibitem[BS06]{BS06} U. Bader, Y. Shalom, \textit{Factor and normal subgroup theorems for lattices in products of groups}. Invent. Math. 163 (2006), 415-454


\bibitem[BV93]{BV93} M. E. B. Bekka and A. Valette, \textit{Kazhdan's property (T) and amenable representations},  Math. Z. 212 (1993), no. 2, 293Ð299.



\bibitem[DS02]{DS02} S. G. Dani, \textit{On conjugacy classes of closed subgroups and stabilizers of Borel actions of Lie groups}. Ergodic Theory Dynam. Systems 22 (2002), no. 6, 1697--1714.

\bibitem[DM11]{DM11} A. Dudko and K. Medynets, \textit{On characters of inductive limits of symmetric groups}, Journal of Functional Analysis
Volume 264, Issue 7, 1 April 2013, 1565--1598.

\bibitem[DM12]{DM12} A. Dudko and K. Medynets, \textit{Finite factor representations of Higman-Thompson groups}, to appear at Groups, Geometry, and Dynamics.



\bibitem[GK12]{GK12} R. Grigorchuk and R. Kravchenko, \textit{On the lattice of subgroups of the lamplighter group}. arXiv:1203.5800




\bibitem[Gr11]{Gr11} R. Grigorchuk, \textit{Some topics of dynamics of group actions on rooted trees.}, The Proceedings of the Steklov Institute of Math., v. 273 (2011), 64--175.



\bibitem[GS99]{GS99} V. Golodets and S. D. Sinelshchikov, \textit{On the conjugacy and isomorphism problems for stabilizers of Lie group actions}. Ergodic Theory Dynam. 

\bibitem[GW97]{GW97} E. Glasner and B. Weiss, \textit{KazhdanÕs property T and the geometry of the collection of invariant measures},
Geometry and Functional Analysis, vol. 7 (1997), 917--935.






\bibitem[JS87]{JS87} V. F. R. Jones and K. Schmidt, \textit{Asymptotically Invariant Sequences and Approximate Finiteness}, American Journal of Mathematics, Vol. 109, No. 1 (Feb., 1987), pp. 91-114.

\bibitem[Kro85]{Kro85} P. H. Kropholler, \textit{A note on the cohomology of metabelian groups}, Math. Proc. Cambridge Philos. Soc. 98 (1985), no. 3, 437--445. 

\bibitem[LOS78]{LOS78} J. Lindenstrauss, G.H. Olsen, Y. Sternfeld, \textit{The Poulsen simplex}, Ann. Inst. Fourier (Grenoble) 28 (1978), 91--114.

\bibitem[Ol79]{Ol79} G. H. Olsen, \textit{On simplices and the Poulsen simplex}. Functional analysis: surveys and recent results, II (Proc. Second Conf. Functional Anal., Univ. Paderborn, Paderborn, 1979), pp. 31--52, Notas Mat., 68, North-Holland, Amsterdam-New York, 1980

\bibitem[Sa11]{Sa11} D. Savchuk, \textit{Schreier Graphs of Actions of ThompsonÕs Group F on the Unit Interval and on the Cantor Set}, arXiv: 1105.4017.

\bibitem[Sch84]{Sch84} K. Schmidt, \textit{Asymptotic properties of unitary representations and mixing},  Proc. London Math. Soc. (3) 48 (1984), no. 3, 445Ð460.

\bibitem[SZ94]{SZ94} G. Stuck and R. J. Zimmer, \textit{Stabilizers for ergodic actions of higher rank semisimple groups}, Ann. of Math. (2) 139 (1994), no. 3, 723--747.

\bibitem[Ve10]{Ve10} A. Vershik, \textit{Nonfree Actions of Countable Groups and their Characters}, Zapiski Nauchn. Semin. POMI 378, 5-16 (2010).
English translation: J. Math. Sci. 174, No. 1, 1-6 (2011).


\bibitem[Ve11]{Ve11} A. Vershik, \textit{Totally nonfree actions and infinite symmetric group}. Moscow Math. J. 12, No. 1, 193--212 (2012).


\bibitem[Vo12]{Vo12} Y. Vorobets, \textit{Notes on the Schreier graphs of the Grigorchuk group}. Contemporary Mathematics 567 (2012), 221--248.


\end{thebibliography}
}
\end{document}